\documentclass[12pt,a4paper,leqno]{amsart}
\usepackage{amssymb}

\numberwithin{equation}{section}

\newtheoremstyle{theorem}{3pt}{3pt}%
{\it}
{}
{\bfseries}
{:}
{.5em}
{}
\theoremstyle{theorem}
\newtheorem{theorem}{Theorem}[section]

\newtheorem{proposition}[theorem]{Proposition}
\newtheorem{corollary}[theorem]{Corollary}
\newtheorem{lemma}[theorem]{Lemma}
\newtheorem{definition}[theorem]{Definition}

\newtheoremstyle{example}{3pt}{3pt}%
{}
{}
{\sc}
{:}
{.5em}
{}
\theoremstyle{example}
\newtheorem{example}[theorem]{Example}
\newtheoremstyle{remark}{3pt}{3pt}%
{}
{}
{\sc}
{:}
{.5em}
{}
\theoremstyle{remark}
\newtheorem{remark}{Remark}[section]

\numberwithin{equation}{section}

\newcommand{\thismonth}{\ifcase\month\or
  January\or February\or March\or April\or May\or June\or
  July\or August\or September\or October\or November\or December\fi
  \space\number\year}
\makeatletter
\newcommand{\low}{\@ifnextchar^{}{^{\vphantom x}}}
\newcommand{\high}{\@ifnextchar_{}{_{\vphantom I}}}
\makeatother
\DeclareSymbolFont{script}{U}{eus}{m}{n}
\DeclareSymbolFontAlphabet{\mathscr}{script}
\DeclareMathSymbol{\EuWedge}{0}{script}{"5E}
\DeclareMathAlphabet{\mathrmsl}{OT1}{cmr}{m}{sl}
\newcommand{\rssymb}[2]{\newcommand{#1}{{\mathrmsl{#2}}}}
\newcommand{\calsymb}[2]{\newcommand{#1}{{\mathcal{#2}}}}
\newcommand{\bbsymb}[2]{\newcommand{#1}{{\mathbb{#2}}}}
\newcommand{\lieoper}[2]{\newcommand{#1}{\mathop
  {\mathfrak{#2}\null}\nolimits}}
\newcommand{\oper}[3][n]{\newcommand{#2}{\mathop
  {\mathrm{#3}\null}\ifx n#1\nolimits\else\limits\fi}}
\newcommand{\rsoper}[3][n]{\newcommand{#2}{\mathop
  {\mathrmsl{#3}\null}\ifx n#1\nolimits\else\limits\fi}}
\bbsymb\C{C} \bbsymb\F{F} \bbsymb\HQ{H}\bbsymb\N{N} \bbsymb\Q{Q}
\bbsymb\R{R} \bbsymb\U{U} \bbsymb\V{V} \bbsymb\W{W} \bbsymb\Z{Z}
\bbsymb\bbf{F} \bbsymb\bbk{K} \bbsymb\bbi{I} \bbsymb\bbl{L} \bbsymb\bbo{O}
\bbsymb\bbj{J}
\bbsymb\bby{Y}
\bbsymb\bbp{P}
\bbsymb\bba{A}
\calsymb\cA{A} \calsymb\cB{B} \calsymb\cC{C} \calsymb\cD{D} \calsymb\cE{E}
\calsymb\cF{F} \calsymb\cG{G} \calsymb\cH{H} \calsymb\cI{I} \calsymb\cJ{J}
\calsymb\cK{K} \calsymb\cL{L} \calsymb\cM{M} \calsymb\cN{N} \calsymb\cO{O}
\calsymb\cP{P} \calsymb\cQ{Q} \calsymb\cR{R} \calsymb\cS{S} \calsymb\cT{T}
\calsymb\cU{U} \calsymb\cV{V} \calsymb\cW{W} \calsymb\cX{X} \calsymb\cY{Y}
\calsymb\cZ{Z}
\oper\End{End} \oper\Hom{Hom}                    
\oper\Sym{Sym} \oper\Skew{Skew}
\oper\Aut{Aut}                                   
\oper\GL{GL} \oper\SL{SL}\oper\Symp{Sp}
\oper\CO{CO} \oper\On{O} \oper\SO{SO} \oper\Pin{Pin} \oper\Spin{Spin}
\oper\CU{CU} \oper\Un{U} \oper\SU{SU} \oper\PSU{PSU}
\rsoper\Diff{Diff} \rsoper\SDiff{SDiff}
\lieoper\der{der}                                
\lieoper\gl{gl} \lieoper\sgl{sl}\lieoper\symp{sp}
\lieoper\co{co} \lieoper\so{so} \lieoper\spin{spin}
\lieoper\cu{cu} \lieoper\un{u}  \lieoper\su{su}
\rsoper\Vect{Vect} \rsoper\Ham{Ham}
\def\la#1{\hbox to #1pc{\leftarrowfill}}
\def\ra#1{\hbox to #1pc{\rightarrowfill}}

\newcommand{\norm}[2][]{|\mkern-2mu|#2|\mkern-2mu|
  _{\lower1pt\hbox{${}_{#1}$}}}
\newcommand{\Norm}[2][]{\bigl|\mkern-3mu\bigr|#2\bigr|\mkern-3mu\bigr|
  _{\lower1pt\hbox{${}_{#1}$}}}

\newcommand{\del}{\partial}                 
\rsoper\dimn{dim}                           
\rsoper\grad{grad}                          
\rsoper\kernel{ker}\rsoper\image{im}        
\rsoper\alt{alt}   \rsoper\sym{sym}         
\rsoper\Ad{Ad}     \rsoper\ad{ad}           
\rsoper\CoAd{CoAd} \rsoper\coad{coad}       
\rsoper\trace{tr}  \rsoper\trfree{tf}       
\rsoper\detm{det}                           
\rsoper\Vol{Vol}                            
\rsoper\divg{div}                           
\rsoper\sign{sign}                          
\rssymb\iden{id}                            
\rssymb\vol{vol}                            
\oper\Imag{Im}\oper\Real{Re}                
\newcommand{\sd}{{\raise1pt\hbox{$\scriptscriptstyle +$}}}
\newcommand{\asd}{{\raise1pt\hbox{$\scriptscriptstyle -$}}}
\newcommand{\sdasd}{{\raise1pt\hbox{$\scriptscriptstyle\pm$}}}
\newcommand{\asdsd}{{\raise1pt\hbox{$\scriptscriptstyle\mp$}}}

\rsoper\scal{scal}
\def\kahl/{k\"ahler}
\def\Kahl/{K{\"a}hler}

\begin{document}
\title[Generalized CoK\"ahler and Generalized K\"ahler Geometry]{Generalized Cok\"ahler Geometry and an Application to Generalized K\"ahler Structures}
\author{Ralph R. Gomez}
\address{ Department of Mathematics and Statistics\\
		Swarthmore College \\
		Swarthmore, PA 19081}
		
\author{Janet Talvacchia}
\address{ Department of Mathematics and Statistics\\
		Swarthmore College \\
		Swarthmore, PA 19081}

\date{\today}

\begin{abstract}
In this paper, we propose a generalization of classical coK\"ahler geometry from the point of view of generalized contact metric geometry. This allows us to generalize a theorem of M.Capursi and I.Goldberg (\cite{[C]},\cite{Gol}) and show that the product $M_{1}\times M_{2}$ of generalized contact metric manifolds $(M_i, \Phi_i,E_{\pm,i}, G_i)$, $ i=1, 2$,
where $M_{1}\times M_{2}$ is endowed with the product (twisted) generalized complex structure induced from $\Phi_1$ and $\Phi_2$, is (twisted) generalized K\"ahler if and only if $(M_i, \Phi_i, E_{\pm,i}, G_i) ,\  \  i=1,2$ are (twisted) generalized coK\"ahler structures.
As an application of our theorem we construct new examples of twisted generalized K\"ahler structures on manifolds that do not admit a classical K\"ahler structure and we give examples of twisted generalized coK\"ahler structures on manifolds which do not admit a classical coK\"ahler structure.
\end{abstract}

\maketitle

\section{Introduction}\label{S:intro}
\indent Consider the product of two manifolds $M_{1}\times M_{2}$ where each manifold has an almost contact structure $(\phi_i,\xi_i,\eta_i)$ on $M_i$, $\phi_i$ is an endomorphism of $TM_i$, $\xi_i$ is a vector field and $\eta_i$ is a 1-form such that $\eta_{i}(\xi_{i})=1$, $i=1,2$. A natural question to ask is what further conditions are needed
on $M_{i}$ to ensure the product $M_{1}\times M_{2}$ is a complex manifold. Morimoto answered this question in \cite{[M]}. He constructed a natural almost complex structure $J$ on the product given by
\begin{equation}
J(X,Y)=(\phi_{1}(X)-\eta_{2}(Y)\xi_{1},\phi_{2}(Y)+\eta_{1}(X)\xi_{2})
\end{equation}
where $X\in TM_{1}$, $Y\in TM_{2}$, and showed that $J$ is integrable if and only if each factor manifold $M_{i}$ is normal almost contact. An interesting corollary of this is that the product of two odd-dimensional spheres is complex. One can further ask under what conditions on $M_{i}$ the product manifold is K\"ahler. It is tempting to speculate that the answer should be that each factor manifold $M_{i}$ is Sasakian. But that cannot be the case since this would mean the product of two Sasakian spheres would be K\"ahler violating Calabi and Eckmann's result that the product of odd dimensional spheres is complex and nonK\"ahler \cite{Cala}.
It was Goldberg \cite{Gol} and later Capursi \cite{[C]} who realized the right geometric structure on the factor manifolds $M_{i}$ should be a coK\"ahler structure. A coK\"ahler manifold is a normal almost contact metric manifold $(\phi,\xi,\eta,g)$ where $g$ is a Riemannian metric compatible with the other structure tensors such that both the one-form $\eta$ is closed and the fundamental two-form $\Omega=g(X,\phi Y)$ is closed, for any sections $X,Y$ of $TM$. (See section 2 for precise definitions.)
We state the theorem of Capursi and Goldberg here since the principal aim of this paper is to generalize their theorem:
\begin{theorem}[Capursi,Goldberg]
Let $(M_{i},\phi_{i},\xi_{i},\eta_{i},g_{i})$, $i=1,2$, be almost contact metric manifolds and let $J$ be defined as above. The manifold $(M_{1}\times M_{2},J,G)$ is K\"ahler if and only if $(M_{i},\phi_{i},\xi_{i},\eta_{i},g_{i})$ is coK\"ahler, where $G=g_{1}+g_{2}.$
\end{theorem}
\indent The beginnings of coK\"ahler geometry started in its geometric role in time-dependant mechanics (see for example \cite{Barb}) and since then the subject has grown in spurts. Various results have emerged over the years, elucidating both the differential geometry and the topology of the manifolds admitting such structures. The reader should refer to \cite{[CDY]} for a recent and comprehensive survey of coK\"ahler geometry, and more generally, cosymplectic geometry.\\
\indent The notion of a generalized complex structure and its twisted counterparts, introduced by Hitchin in his paper \cite{[H]} and developed by Gualtieri (\cite{[G1]},\cite{[G2]}) is a framework that unifies both complex and symplectic structures.  These structures exist only on even dimensional manifolds. The odd dimensional analog of this structure, a generalized contact structure, was taken up by Vaisman (\cite {[V1]},\cite{[V2]}),  Poon, Wade \cite{[PW]}, and Sekiya \cite{[S]}. This framework unifies almost contact, contact, and cosymplectic structures.  Generalized K\"ahler structures were introduced by Gualtieri \cite{[G1]},\cite{[G2]},\cite{[G3]} and have already found their way into the physics literature (\cite{[Hu]}, \cite{[LMTZ]}, \cite{[LRUZ]}, \cite{[Gr]}).\\
\indent In order to prove a generalized contact version of Theorem 1.1, the first step is to reformulate Morimoto's theorem
discussed above in the generalized contact setting. This step was accomplished by the authors in \cite{[GT]}. Consider the generalized almost contact structure $(\Phi_{i},E_{\pm,i})$ where $\Phi_{i}$ is an endomorphism of $TM_{i}\oplus T^{*}M_{i}$ and
$E_{\pm,i}$ are sections of $TM_{i}\oplus T^{*}M_{i}$, $i=1,2$, such that the conditions given in (2.3)-(2.5) are satisfied. It was shown by the authors that one can
generalize equation (1.1) and this generalized almost complex structure
on $M_{1}\times M_{2}$ is given by:
\begin{align}
\mathcal J (X_1 +\alpha_1,X_2 +\alpha_2 ) = &( \Phi_1(X_1+\alpha_1) - 2\langle E_{+,2},X_2 +\alpha_2 \rangle E_{+,1} - 2\langle E_{-,2}, X_2 +\alpha_2 \rangle E_{-,1}, \cr
 &\Phi_2(X_2 +\alpha_2)  + 2\langle E_{+,1}, X_1 +\alpha_1 \rangle E_{+,2} + 2\langle E_{-,1}, X_1 +\alpha_1 \rangle E_{-,2} )
\end{align} for any sections $X_{i}+\alpha_{i}$ of $TM_{i}\oplus T^{*}M_{i}$.
This formula was then used to give a proof of a generalization almost contact version of Morimoto's theorem mentioned above.
\begin{theorem}\cite{[GT]}
Let $M_1$ and $M_2$ be odd dimensional smooth manifolds each with generalized almost contact structures $(\Phi_i, E_{\pm,i})$ $i=1,2$.  Then $M_1\times M_2$ admits a generalized almost complex structure $\mathcal J$. Further $\mathcal J$ is a generalized complex structure if and only if both $(\Phi_i, E_{\pm,i})\  \   i=1,2$ are strong generalized contact structures and $[[E_{\pm, i},E_{\mp ,i}]] = 0$.
\end{theorem}

\indent In this article, we will first propose a generalization of coK\"ahler geometry using the language of generalized contact metric geometry and then we prove the following generalization of Theorem 1.1:
\begin{theorem}\label{T1}
Let $M_{1}$ and $M_{2}$ be odd dimensional smooth manifolds each with a (twisted) generalized contact metric structure
$(\Phi,E_{\pm,i},G_i),i=1,2$ such that on the product $M_1\times M_2$ are two
(twisted) generalized almost complex structures: $\mathcal{J}_1$ which is the natural generalized almost complex structure induced from $\Phi_1$ and $\Phi_2$ and $\mathcal{J}_2=G\mathcal{J}_1$ where
$G=G_1\times G_2$. Then $(M_1\times M_2,\mathcal{J}_1,\mathcal{J}_2)$ is (twisted) generalized K\"ahler if and
only if $(\Phi_{i},E_{\pm,i},G_i)$, $i=1,2$ are (twisted) generalized coK\"ahler structures.
\end{theorem}

\indent In section 2 we gather the basics of coK\"ahler geometry, generalized complex, generalized K\"ahler, and generalized contact geometry. In section 3, we state and prove some basic properties of generalized almost contact metric structures and define the notion of a generalized coK\"ahler structure. Then, in section 4 we prove Theorem 1.3. In section 5, we give numerous examples of generalized coK\"ahler structures. In particular, we are able to construct almost K\"ahler nonK\"ahler manifolds that admit twisted generalized K\"ahler structures and
we are able to construct almost coK\"ahler noncoK\"ahler manifolds that admit twisted generalized coK\"ahler structures.

\section{Preliminaries}\label{B:Back}
In this section we will record some of the fundamental geometric structures needed for the generalization of coK\"ahler geometry. But first, it will be worthwhile to recall the formal definition of a coK\"ahler structure on a manifold.\\
\indent An almost contact metric structure on $M$ is given by the the following structure tensors $(\phi,\xi, \eta, g)$ where $\phi$ is a $(1,1)$ tensor field, $\xi$ is a vector field and $\eta$ is a $1$-form, given by the following conditions

\begin{equation}
\phi^{2}=-I+\eta\otimes\eta, \hspace{.5cm}\eta(\xi)=1
\end{equation}
and where $g$ is a Riemannian metric subject to the following compatibility condition
\begin{center}
$g(\phi X, \phi Y)=g(X,Y)-\eta(X)\eta(Y)$
\end{center}
for an vector fields $X,Y\in \Gamma(TM)$. We can use the Riemannian metric $g$ and $\phi$ to construct the fundamental 2-form
\begin{equation}
\Omega(X,Y)=g(X,\phi Y).
\end{equation}

Finally, recall the Nijenhuis torsion tensor $$N_{I}(X,Y)=[ IX,IY]+I^{2}[X,Y]-I[X, IY]-I[IX,Y]$$ is defined for any $(1,1)$ tensor field $I.$ An almost contact (metric) structure is normal if $N_{\phi}=-2\xi\otimes d\eta$ \cite{[CDY]}. Equivalently, an almost contact (metric) structure on $M$ is normal if
the associated almost complex structure coming from (1.1) on $M\times \mathbb{R}$ is integrable. We are now ready for the definition of a coK\"ahler structure.
\begin{definition}
An almost coK\"ahler manifold is an almost contact metric manifold $(M,\phi,\xi,\eta,g)$ such that the fundamental 2-form $\Omega$
and the 1-form $\eta$ are closed. Furthermore, $M$ is a coK\"ahler manifold if the underlying almost contact structure is normal, that is
$N_{\phi}=-2\xi\otimes d\eta=0.$

\end{definition}

We give here some examples coK\"ahler manifolds which will be useful in section 5.
\begin{example}
On $S^{1}$ we can define $(\phi,\xi,\eta,g)$ where $\phi=0$, $\xi=\del_t$ where $t$ is the coordinate on $S^1$, $\eta=dt$, and
$g=dt\otimes dt.$ It is clear that $S^1$ is almost coK\"ahler. The normality follows at once since $\phi=0$ and $\eta$ is closed.
The manifold $S^1$ together with this coK\"ahler structure will be called the trivial coK\"ahler structure on $S^{1}.$

\end{example}

\begin{example}\cite{[CDY]}
Let $(N,J,g)$ be an almost K\"ahler manifold and form the product, denoted by $M$, with $\mathbb{R}$ (or $S^{1}$). Let $t$ denote the coordinate on $\mathbb{R}$
and let $(X,f\del_{t})$ be a vector field on $M$, where $f$ is any smooth function on the product. Now define an endomorphism of $TM$
by letting $\phi(X,f\del_{t}):=(JX,0)$. Moreover, define $\xi:=\del_{t}$ and a 1-form $\eta=dt.$ The metric on $M$ can be taken to be the product metric $h=g+dt^{2}.$ A straightforward calculation shows that $(M,\phi,\xi,\eta,h)$ is coK\"ahler if and only if $(N,J,g)$ is K\"ahler.
\end{example}
\indent Just as a K\"ahler structure imposes topological restrictions on the manifold, a coK\"ahler structure on an odd dimensional manifold imposes some topological restrictions as well. For example,
it was shown in \cite{[BlGo]} that all the Betti numbers of a compact coK\"ahler manifold are non-zero. Here we state another topological result that we will use in the last section when we construct examples.
\begin{theorem}\cite{[Chn]}
Let $M$ be a compact coK\"ahler manifold. Then the first Betti number of $M$ is odd.
\end{theorem}

\indent We now move to a very brief review of generalized geometric structures. Throughout this paper we let $M$ be a smooth manifold. Consider the big tangent bundle $TM\oplus T^*M$.  We define a neutral metric on $TM\oplus T^*M$ by$$  \langle X + \alpha  , Y + \beta \rangle =  \frac{1}{2} (\beta (X) + \alpha (Y) )$$ and the (H-twisted) Courant bracket by $$[[X+\alpha, Y+ \beta ]]_{H} = [X,Y] + {\mathcal L}_X\beta -{\mathcal L}_Y\alpha -\frac{1}{2} d(\iota_X\beta - \iota_Y\alpha)+\iota_{Y}\iota_{X}H$$ where $X, Y \in TM$ and $\alpha ,\beta  \in T^*M$ and $H$ is a real closed 3-form. A subbundle of $TM\oplus T^*M$ is said to be involutive if its sections are closed under the ($H$-twisted) Courant bracket\cite{[G1]}.

\begin{definition}
A generalized  almost complex structure on $M$ is an endomorphism $\mathcal J$ of $TM\oplus T^*M$ such that $\mathcal J + \mathcal J^* =  0 $ and $\mathcal J^2 = - Id$. If the $\sqrt{-1}$ eigenbundle $L\subset (TM\oplus TM^{*})\otimes\mathbb{C}$ associated to $\mathcal J$ is involutive with respect to the
(H-twisted) Courant bracket, then $\mathcal J$ is called an (H-twisted) generalized complex structure.
\end{definition}

\indent Here are the prototypical examples in the case when $H=0$:

\begin{example}  \cite {[G1]}
Let $(M^{2n}, J)$ be a complex structure.  Then we get a generalized complex structure by setting
$$\mathcal J_{J} = \left ( \begin{array}{cc}  -J & 0 \\ 0 & J^* \end{array} \right ).$$
\end{example}

\begin{example}  \cite {[G1]}
Let $(M^{2n}, \omega )$ be a symplectic structure.  Then we get a generalized  complex structure by setting
$$\mathcal J_{\omega} = \left ( \begin{array}{cc}  0 & -\omega^{-1} \\ \omega & 0 \end{array} \right ).$$
\end{example}
\indent Diffeomorphisms of $M$ preserve the Lie bracket of smooth vector fields and in fact such diffeomorphisms are the
only automorphisms of the tangent bundle. But in generalized geometry, there is actually more flexibility. That is to say,
given $T\oplus T^{*}$ equipped with the (twisted) Courant bracket, the automorphism group is comprised of the diffeomorphisms of $M$
and some additional symmetries called \emph{B}-field transformations \cite{[G1]}.
\begin{definition}\cite{[G1]}
Let $B$ be a closed two-form which we view as a map from $T \rightarrow T^{*}$ given by interior product. Then the invertible bundle map
$$e^{B}:= \left ( \begin{array}{cc}  1 & 0 \\ B & 1 \end{array} \right):X+\xi \longmapsto X+\xi + \iota_{X}B$$
is called a B-field transformation.
\end{definition}
A $B$-field transformation of a (twisted) generalized (almost) complex structure $( M, e^B {\mathcal J} e^{-B})$ is again a (twisted) generalized (almost) complex structure.

\indent Recall that we can reduce the structure group of $T\oplus T^{*}$ from $O(2n,2n)$ to the maximal compact subgroup $O(2n)\times O(2n)$. This
is equivalent to an orthogonal splitting of $T\oplus T^{*}=V_{+}\oplus V_{-}$, where $V_{+}$ and
$V_{-}$ are positive and negative definite respectfully with respect to the inner product. Thus we can define a positive definite Riemannian metric
on the big tangent bundle by $$G=<,>|_{V_{+}}-<,>|_{V_{-}}.$$
A positive definite metric $G$ on $M$ is an automorphism of $TM\oplus T^*M$ such that $G^{*}=G$ and $G^{2}=1.$ In the presence
of a generalized almost complex structure $\mathcal J_1$, if $G$ commutes with $\mathcal J_1$ ( $G\mathcal J_1 = \mathcal J_1  G$) then $G\mathcal J_1 $ squares to $-1$ and we generate a second generalized almost complex structure, $\mathcal J_2$ $= G\mathcal J_1$, such that $\mathcal J_1$ and $\mathcal J_2$ commute and $G=-\mathcal J_1 \mathcal J_2$.
We are now able to recall the following:
\begin{definition}\cite{[G1]}An (H-twisted) generalized K\"ahler structure is a pair of commuting (H-twisted) generalized complex structures $\mathcal J_{1}, \mathcal J_{2}$ such that
$G=-\mathcal J_{1}\mathcal J_{2}$ is a positive definite metric on $T\oplus T^{*}.$
\end{definition}
The two examples just given together give the standard example of a generalized K\"ahler manifold in the case $H=0$ \cite{[G1]}.
\begin{example}
Consider a K\"ahler structure $(\omega,J,g)$ on $M$. By defining $\mathcal J_{J}$ and $\mathcal J_{\omega}$ as in
Examples 2.6 and 2.7, we obtain a generalized K\"ahler structure on $M$, where
$$G=\left ( \begin{array}{cc}  0 & g^{-1} \\ g & 0 \end{array} \right ).$$

\end{example}
\indent Let us now recall the odd dimensional analog of generalized complex geometry. We use the definition given in \cite{[S]}.
\begin{definition} A generalized almost contact structure on $M$ is a triple $(\Phi, E_\pm)$ where $\Phi $ is an endomorphism of $TM\oplus T^*M$, and $E_+$ and $E_-$ are sections of $TM\oplus T^*M$ which satisfy
\begin{equation}
\Phi + \Phi^{*}=0
\end{equation}
\begin{equation}\label{phi}
\Phi \circ \Phi = -Id + E_+ \otimes E_- + E_- \otimes E_+
\end{equation}
\begin{equation}\label{sections}
 \langle E_\pm, E_\pm \rangle = 0,  \  \    2\langle E_+, E_-  \rangle = 1.
\end{equation}

\end{definition}
An easy and immediate consequence \cite{[S]} of these definitions is
\begin{equation}\label{PhivanishE}
\Phi(E_{\pm})=0.
\end{equation}
Now, since $\Phi$ satisfies $\Phi^3 + \Phi =0$, we see that $\Phi$ has $0$ as well as $\pm \sqrt{-1}$ eigenvalues when viewed as an endomorphism of the complexified big tangent bundle $(TM\oplus T^*M) \otimes { \mathbb C}$.  The kernel of $\Phi$ is $L_{E_+} \oplus L_{E_-}$ where $L_{E_\pm}$ is the line bundle spanned by ${E_\pm}$.  Let $E^{(1,0)}$ be the $\sqrt{-1}$ eigenbundle of $\Phi$.  Let $E^{(0,1)}$ be the $-\sqrt{-1}$ eigenbundle. Observe:

$$
E^{(1,0)} = \lbrace X + \alpha - \sqrt{-1}  \Phi (X + \alpha ) |  \langle E_\pm, X + \alpha \rangle = 0 \rbrace
$$

$$
E^{(0,1)} = \lbrace X + \alpha + \sqrt{-1}  \Phi (X + \alpha ) | \langle E_\pm, X + \alpha \rangle = 0 \rbrace .$$

Then the complex line bundles
$$L^+ = L_{E_+} \oplus E^{(1,0)}$$
and
$$L^- = L_{E_-} \oplus E^{(1,0)}$$
are maximal isotropics.
\begin{definition}
A generalized almost contact structure $(\Phi,E_{\pm})$ is an (H-twisted) generalized contact structure if either $L^{+}$ or $L^{-}$ is closed with respect to the (H-twisted) Courant bracket. The (H-twisted) generalized
contact structure is strong if both $L^{+}$ and $L^{-}$ are closed with respect to the (H-twisted) Courant bracket.
\end{definition}
Here are the standard examples in the untwisted case $H=0:$

\begin{example}  \cite {[PW]}
Let $(\phi , \xi, \eta)$ be a normal almost contact structure on a manifold $M^{2n+1}$.  Then we get a generalized almost contact structure by setting
$$ \Phi = \left ( \begin{array}{cc}  \phi & 0 \\ 0 & -\phi^* \end{array} \right ),\  \   E_+ = \xi,\  \   E_-= \eta $$  where $(\phi^*\alpha )(X) = \alpha (\phi (X)), \   X \in TM,\   \alpha \in T^{*}M$. Moreover, $(\Phi, E_\pm)$ is an example of a strong generalized almost contact structure.
\end{example}

\begin{example}  \cite {[PW]}
Let $( M^{2n+1}, \eta )$ be a contact manifold with $\xi $ the corresponding Reeb vector field so that
$$ \iota_\xi d\eta = 0 \  \  \  \eta ( \xi ) = 1.$$
Then $$\rho ( X) := \iota_X d\eta - \eta ( X)\eta$$ is an isomorphism from the tangent bundle to the cotangent bundle.  Define a bivector field by
$$\pi (\alpha , \beta ) := d\eta (\rho^{-1}( \alpha ), \rho^{-1}( \beta )),$$
where $\alpha, \beta \in T^{*}. $
We obtain a generalized almost contact structure by setting
$$ \Phi = \left ( \begin{array}{cc}  0 & \pi \\ d\eta & 0 \end{array} \right ),\  \   E_+ = \eta,\  \   E_-= \xi .$$
In fact, $(\Phi, E_{\pm})$ is an example which is not strong.
\end{example}

\section{ The Definition of Generalized CoK\"ahler and Some Properties and Examples}
\indent The classical notions of normal almost contact structures, contact metric structures, and cosymplectic structures all have analogs in the generalized context. Until now, the notion of a generalized coK\"ahler structure has not been defined. In this section we propose a definition of a generalized coK\"ahler  structure.\\
\indent Recall an almost contact structure $(\phi, \xi, \eta )$ on an odd dimensional manifold $M$ is called normal if the associated almost complex structure on $M \times {\mathbb R}$ is integrable.  In our previous paper \cite{[GT]} (see also the introduction), we proved that the product of generalized almost contact structures $(M_i, \Phi_i, E_{\pm,i})$, $ i=1,2$ is a generalized complex structure if and only if each $\Phi_i$ is strong and $[[E_{\pm,i}, E_{\mp,i}]] = 0$.  Thus, in keeping with the classical notion of normal almost contact structure, we have
\begin{definition}
A generalized almost contact structure $(M, \Phi,E_{\pm})$ is a normal generalized contact structure
if $\Phi$ is strong with respect to the (H-twisted) Courant bracket and $[[E_+, E_-] ]_{H}= 0.$
\end{definition}
As a consequence of the theorem in our previous paper \cite {[GT]} if we take $M$ to have a normal generalized contact structure
and for $\mathbb{R}$ to have the trivial normal generalized contact structure $(\Phi=0, E_{+}=dt, E_{-}=\partial t)$, then the cone
$M\times \mathbb{R}$ admits a generalized complex structure. Thus our definition is consistent with the more restrictive definition of a normal generalized almost contact structure given in \cite{[V1]}.

\indent An almost contact metric structure on $M^{2n+1}$ is an almost contact structure $(\phi , \xi , \eta )$ and a Riemannian metric $g$ that satisfies $g(\phi X, \phi Y) = g(X,Y) -\eta (X) \eta (Y)$. Sekiya \cite{[S]} defined a generalized almost contact metric structure as a generalized almost contact structure
$(\Phi , E_{\pm})$ along with a generalized Riemannian metric $G$ that satisfies
\begin{equation}\label{compatcond}
-\Phi G \Phi = G - E_+ \otimes E_+ -E_- \otimes E_-.
\end{equation}

We give here a lemma regarding generalized almost contact metric structures that will be useful in what follows.

\begin{lemma}
Let $(\Phi, E_{\pm}, G)$ be a (twisted) generalized almost contact metric structure on $M^{2n+1}.$ Then the following statements hold:
\begin{enumerate}
\item $G(E_\pm) = E_\mp$
\item $G\Phi = \Phi G$
\item $G(E^{(1,0)})= E^{(1,0)}$
\item $(e^B\Phi e^{-B},e^{B}E_{\pm},e^{B}G e^{-B})$ is a (twisted) generalized almost contact metric structure, where
$B$ is a B-field transformation. Furthermore, this (twisted) generalized almost contact metric structure is strong if $(\Phi, E_{\pm})$ is strong.
\end{enumerate}

\end{lemma}

\begin{proof}
Since $(\Phi, E_{\pm},G)$ is a generalized almost contact metric structure we have
$$0= -\Phi G \Phi(E_+)  = G(E_+) - E_+ \otimes E_ +(E_+)  - E_- \otimes E_-(E_+)  = G(E_+) - E_-$$
and so $G(E_+) = E_-.$ Similarly one shows $G(E_-) = E_+$.\\
\indent For property (ii), recall we have,

$$-\Phi G \Phi = G - E_+ \otimes E_+ - E_- \otimes E_- .$$ Apply $\Phi$ to both sides getting
$$-\Phi^2 G \Phi = \Phi G.$$
Using the formula (2.2) for $\Phi^2$, we get
$$G\Phi - E_+ \otimes E_- \circ G\Phi - E_- \otimes E_+\circ G\Phi  = \Phi G $$
But Lemma 1 in \cite{[GT]} gives that $E_\pm \circ \Phi = 0$.
This combined with (i) and the self-adjointness of $G$ then imply
$$G\Phi = \Phi G.$$

To establish property (iii) we show first that $E^{(1,0)} \subset G(E^{(1,0)})$. Let $Y+\beta \in E^{(1,0)}$.  Then $Y+ \beta = X + \alpha - \sqrt{-1} \Phi ( X+ \alpha )$ for some $X+ \alpha \in TM\oplus T^*M $ such that $<X + \alpha , E_\pm> = 0$.
By property (i) and the fact that $G$ is self-adjoint, we obtain

$$0 = < G(X + \alpha ), E_\pm>.$$

Now consider  $$G(X + \alpha ) - \sqrt{-1} \Phi G ( X + \alpha) \in E^{(1,0)}.$$
By applying $G$ again and using the fact that $G^2 = {\rm Id}$ and $\Phi $ and $G $ commute gives the first inclusion.

To show inclusion in the other direction, let $Y+ \beta \in G(E^{(0,1)})$.  Then $Y + \beta = G( X + \alpha - \sqrt{-1} \Phi ( X + \alpha )) $ for some $ X + \alpha \in TM \oplus T^*M$ such that $< X+ \alpha, E_\pm>= 0$.
But,
$$0 = < X + \alpha, E_\pm> = < X+ \alpha , G(E_\mp )> =<G(X+ \alpha ), E_\mp >$$
since $G$ is self-adjoint. Thus, $G(X+\alpha ) - \sqrt{-1} \Phi G(X+\alpha ) \in E^{(1,0)}$.
Since $\Phi$ and $G$ commute, $ Y+\beta = G(X+\alpha )- \sqrt{-1} G \Phi (X+ \alpha ) \in E^{(1,0)}$.

For property  (iv),
Sekiya \cite{[S]} showed that $(e^B\Phi e^{-B},e^{B}E_{\pm})$ is again a (twisted) generalized almost contact structure. It remains to show that
$e^{B}Ge^{-B}$ satisfies the compatibility condition 3.1 with sections $e^{B}E_{\pm}$. This reduces to showing that


$$e^{B}(E_{\pm}\otimes E_{\pm})e^{-B}= e^{B}E_{\pm}\otimes e^{B} E_{\pm}.$$
Let $X+\alpha \in T\oplus T^{*}$ and so
$$(e^{B}(E_{+}\otimes E_{+})e^{-B})(X+\alpha )=E_{+}(X+\alpha -\iota_{X}B)e^{B}E_{+}=(E_{+}(X+\alpha)-\iota_{\xi_{+}}\iota_{X} B)e^BE_{+}$$
where we have used that $E_{+}=\xi_{+}+\eta_{+}.$
On the other hand, we have
\begin{align*}
&e^{B}E_{+}\otimes e^{B} E_{+}(X+\alpha)=(e^BE_{+})(X+\alpha)e^{B}E_+=(E_{+}+\iota_{\xi_+}B)(X+\alpha)e^{B}E_+\\
&=(E_{+}(X+\alpha)+\iota_{X}\iota_{\xi_{+}}B)e^{B}E_+ =(E_{+}(X+\alpha)-\iota_{\xi_+}\iota_{X}B)e^{B}E_+\\
\end{align*}
where we have used the general property that $\iota_{X}\iota_{Y}=-\iota_{Y}\iota_{X.}$ A similar argument is used to show
$$e^{B}(E_{-}\otimes E_{-})e^{-B}= e^{B}E_{-}\otimes e^{B} E_{-}.$$
The strong property follows immediately since the (twisted) Courant bracket is invariant under $B$-field transforms. Hence
$L^{\pm}$ being (twisted) Courant involutive is preserved as well. (See also Proposition 3.42 in \cite{[G1]}.)

\end{proof}

\begin{remark} Observe that an easy consequence of Lemma 3.2  (i) and (ii) together with (3.1) is that $(M, G\Phi, GE_{\pm}=E_{\mp},G)$ is again a generalized almost contact metric structure.
\end{remark}
Now that all of the pieces are in place, we are ready for our definition of a generalized coK\"ahler structure.
\begin{definition} A normal generalized contact metric structure $(M, \Phi, E_{\pm}, G)$ is (H-twisted) generalized coK\"ahler if $G\Phi$
is also strong with respect to the (H-twisted) Courant bracket.
\end{definition}
\begin{remark}
The sections associated to $G\Phi$ are $GE_{\pm}=E_{\mp}$ and so automatically we get that $[[E_{\pm},E_{\mp}]]=0$ for the generalized contact metric structure associated with $G\Phi$.  Hence,  we could have
alternatively defined a coK\"ahler structure to be a generalized contact metric structure $(M, \Phi , E_\pm , G)$ such that both $(M, \Phi,E_{\pm} , G)$ and $(M, G\Phi,E_{\mp} , G)$ are normal.
\end{remark}

If $(M, \Phi, E_\pm, G)$ is a normal generalized contact metric structure then by definition $\Phi $ is strong. It is important to emphasize that there may be normal generalized contact metric structures where $G\Phi$ is not strong as the following example shows.

\begin{example} Let  $ M=SU(2)$.
On the Lie algebra $su(2)$ choose a basis $\lbrace X_1,  X_2,  X_ 3 \rbrace $ and a dual basis $\lbrace  \sigma^1, \sigma^2, \sigma^  3 \rbrace$ such that $[ X_i, X_j ] = -X_k$ and $ d\sigma^i = \sigma^j \wedge \sigma^k$ for cyclic permutations of $\lbrace i, j, k \rbrace$.  One can construct a classical normal almost contact structure by taking $\phi = X_2 \otimes \sigma^1 - X_1 \otimes \sigma^2$, $ \xi = X_3$ and $\eta = \sigma^3$.  Then, as in Example 2.13, we can construct a generalized almost contact structure by letting
$$ \Phi = \left ( \begin{array}{cc}  \phi & 0 \\ 0 & -\phi^* \end{array} \right ),\  \   E_+ = X_{3},\  \   E_-= \sigma^{3} $$  where $(\phi^*\alpha )(X) = \alpha (\phi (X)), \   X \in TM,\   \alpha \in T^{*}M$.  One computes easily that $E^{(1,0)}_\phi = {\rm span} \lbrace X_1- \sqrt{-1}X_2 ,  \sigma^1- \sqrt{-1} \sigma^2 \rbrace$ so that $L^+ = {\rm span} \lbrace X_3, X_1 - \sqrt{-1}X_2 ,  \sigma^1- \sqrt{-1} \sigma^2 \rbrace$ and $L^{-} =  {\rm span} \lbrace \sigma^3, X_1 - \sqrt{-1}X_2 ,  \sigma^1 -\sqrt{-1} \sigma^2 \rbrace$.  For $L^{+}$, the relevant Courant brackets give
$$[[X_1- \sqrt{-1}X_2,\sigma^1- \sqrt{-1} \sigma^2]]=0, \hspace{.4cm} [[X_{3},\sigma^{1}-\sqrt{-1}\sigma^{2}]]=-\sqrt{-1}(\sigma^{1}-\sqrt{-1}\sigma^{2})$$
as well as $[[X_{3}, X_{1}-\sqrt{-1}X_{2}]]=\sqrt{-1}(X_{1}-\sqrt{-1}X_2).$ Similarly, for $L^{-}$ the relevant Courant bracket is
$$[[\sigma^{3},X_{1}-\sqrt{-1}X_2]]=\sqrt{-1}(\sigma^{1}-\sqrt{-1}\sigma^{2}).$$ Since $(\phi,\xi,\eta)$ is normal, we have that $[[E_{+},E_{-}]]=\mathcal{L}_{X_3}\sigma^{3}=0.$ Thus $(\Phi,E_{\pm})$ is a normal generalized contact structure.\\
\indent Now, define a generalized metric $G$ on $TM\oplus T^*M$ by $\left ( \begin{array}{cc}  0 & g^{-1} \\ g & 0 \end{array} \right )$ where $g$ is any Riemannian metric compatible with the almost contact structure. It is a straightforward calculation to verify $G$ is compatible with the $(\Phi,E_{\pm})$. So we have that $G\Phi$  defines a generalized almost contact structure on $SU(2)$. But observe that $$L^+_{G\Phi } =  {\rm span} \lbrace \sigma^3 , X_1 - \sqrt{-1} \sigma^2 , X_2  + \sqrt{-1} \sigma^1 \rbrace$$ and hence $[[ X_1 - \sqrt{-1} \sigma^2, X_2 + \sqrt{-1} \sigma^1]] = -X_3 \notin L^+_{G\phi}$. Therefore, $G\Phi$ is not strong even though $\Phi$ is strong.
\end{example}
\begin{remark}
If we instead use an H-twisted Courant bracket twisted by the real closed three form $H=\sigma^{1}\wedge\sigma^{2}\wedge\sigma^{3}$ then one
can similarly show that $\Phi$ is strong and $G\Phi$ is not strong.
\end{remark}

Recall that a K\"ahler structure on $M$ induces a generalized K\"ahler structure on $M$. The odd dimensional version of this holds as well:
\begin{proposition} Any coK\"ahler manifold is generalized coK\"ahler.
\end{proposition}
\begin{proof}
Let $(\phi, \xi, \eta, g)$ be a coK\"ahler structure on $M$. Let $\pi$ be the bivector field as defined in Example 2.14 and let $\Omega$ be the fundamental two form as given in (2.2).
Define $$\Phi_{\phi} = \left ( \begin{array}{cc}  \phi & 0 \\ 0 & -\phi^{*} \end{array} \right ),  \Phi_{\Omega} = \left ( \begin{array}{cc}  0 & \pi^{\sharp} \\ \Omega^{\flat} & 0 \end{array} \right ),  G = \left ( \begin{array}{cc}  0 & g^{-1} \\ g & 0 \end{array} \right ).$$
We will argue that $(\Phi_{\phi},E_{\pm}, G)$ is a generalized coK\"ahler structure on $M.$ Let us first verify that such a $G$ is compatible with $(\Phi_{\phi},E_{+}=\xi,E_{-}=\eta)$. The compatibility condition (3.1) for our case reduces to verifying the
following equality:
\begin{equation}
\phi g^{-1}\phi^{*}\alpha + \phi^{*}g\phi X = g^{-1}\alpha+gX-\alpha(\xi)\xi -\eta(X)\eta,
\end{equation}
where $X+\alpha \in T\oplus T^{*}.$
Now, $\phi^{*}g\phi X = g(\phi X, \phi)=g(X,\cdot)-\eta(X)\eta$, using the compatibility of $g.$ Furthermore,
since $g$ is a Riemannian metric which induces an isomorphism between $T$ and $T^{*}$, let us write
$\alpha = g(Y)$ for some $Y\in T.$ Then,
$$\phi g^{-1}\phi^{*}gY=\phi g^{-1}g(Y,\phi)=-\phi g^{-1}g(\phi Y,)=-\phi^{2}Y.$$ But this is precisely,
$g^{-1}\alpha -\alpha(\xi)\xi$, using the fact that $g(Y,\xi)=\eta (Y) $ and $g(X,\phi Y)=-g(\phi X,Y).$

Since we assume $(\phi, \xi, \eta, g)$ is a coK\"ahler structure on $M$, it is normal. Thus,
$[[E_{+},E_{-}]]=[[\xi,\eta]]=0$. Moreover, by  Proposition 3.4 of \cite{[PW]} $(M, \Phi_{\phi},E_{+}=\xi,E_{-}=\eta)$
is strong. Thus we have showed $(M, \Phi_{\phi},E_{+}=\xi,E_{-}=\eta)$ is a normal generalized contact metric structure. Now,
a straightforward computation shows $G\Phi_{\phi}=\Phi_{\Omega}$. Again, in \cite{[PW]}, it was
shown $\Phi_{\Omega}$ is strong. Therefore, $(M, \Phi_{\phi},E_{\pm}, G)$ is a generalized coK\"ahler structure.
\end{proof}

\section{Proof of the Main Theorem}
\indent Before we begin the proof of Theorem 1.3, it will be useful to formulate a twisted version of Theorem 1.2.
Fix the generalized almost contact structures $(M_i, \Phi_i, E_{\pm,i})$, $i=1,2.$ Now on
the product even dimensional manifold $M_{1}\times M_{2}$ we have the generalized almost complex structure given by (1.2) in \cite{[GT]}. We showed in \cite{[GT]} that $(M_{1}\times M_{2}, \mathcal J)$ is a generalized complex structure if and only if $(M_i, \Phi_i, E_{\pm,i})$ are each strong generalized contact structures and $[[E_{\pm,i},E_{\mp,i}]]=0.$ It is easy to see that this theorem trivially extends to the $H$-twisted case. The only new ingredient needed in the proof is how the $\tilde{H}$-twisted Courant bracket behaves on the product $M_{1}\times M_{2}.$ It is given by:

\begin{equation}
[[(A,B), (C,D)]]_{\tilde{H}}=([[A,C]]_{H_1},[[B,D]]_{H_2})
\end{equation}
where $H_{i}$ is a real closed three form on $M_{i}$, ($i=1,2$), $\tilde{H}=H_1+H_2$, $(A,B)$ and $(C,D)$ are sections of the generalized tangent bundle of $M_1\times M_2$. Here is now the twisted
version of Theorem 1.2.
\begin{theorem}\cite{[GT]}
Let $M_1$ and $M_2$ be odd dimensional smooth manifolds each with ($H_{i}$-twisted) generalized almost contact structures $(\Phi_i, E_{\pm,i})$ $i=1,2$.  Then $M_1\times M_2$ admits an ($\widetilde{H}$-twisted) generalized almost complex structure $\mathcal J$. Further $\mathcal J$ is an ($\widetilde{H}$- twisted) generalized complex structure if and only if both $(\Phi_i, E_{\pm,i})\  \   i=1,2$ are strong ($H_i$-twisted) generalized contact structures and $[[E_{\pm, i},E_{\mp ,i}]]_{H_{i}} = 0$.
\end{theorem}

To prove Theorem 1.3, we have to introduce generalized metrics
$G_i$ which are compatible with $\Phi_i$, $i=1,2$. Then on $M_{1}\times M_{2}$, with its product metric $G=(G_{1},G_{2})$, we will have two generalized almost complex structures $\mathcal{J}_{1}$ and $\mathcal{J}_{2}=G\mathcal{J}_1$ that commute. We will prove that
$\mathcal{J}_2$ is integrable if and only if $G_i\Phi_{i}$ are strong.\\
\indent First we record the following lemma which will be useful. This was proven in \cite{[GT]} in the case in which $H=0$ but it is trivial to show that the lemma still holds in the twisted case.
\begin{lemma} \cite{[GT]} Let $(M, \Phi, E_\pm)$ be a (H-twisted) generalized almost contact structure. Then $\Phi $ is strong if and only if  $[[L^+, E^{(1,0)}]]_{H} \subset E^{(1,0)}$ and $[[L^-, E^{(1,0)}]]_{H} \subset E^{(1,0)}$.
\end{lemma}

Now we state the main theorem to be proved.
\begingroup
\def\thetheorem{\ref{T1}}
\begin{theorem}
Let $M_1$ and $M_2$ be odd dimensional manifolds each with a ($H_i$-twisted) generalized coK\"ahler structure $(\Phi_{i},E_{\pm,i},G_i)$ $i=1,2.$
Furthermore, let $\mathcal J_1$ be defined as in (1.2) and let $\mathcal J_2=G\mathcal J_1$ were $G$ is the product metric. Then $(M_1\times M_2,\mathcal J_1,\mathcal J_2)$ is ($\tilde{H}$-twisted) generalized K\"ahler if an only if $(\Phi_i,E_{\pm,i},G_i)$ $i=1,2$ are ($H_i$-twisted) generalized coK\"ahler.
\end{theorem}

\addtocounter{theorem}{-1}
\endgroup
\begin{proof}
Assume $M_1$ and $M_2$ are odd dimensional manifolds each with ($H_i$-twisted) generalized coK\"ahler structure $(\Phi_i, E_{\pm, i}, G_i )$.
On $M_{1}\times M_{2}$ we have the product metric $G=(G_{1},G_{2})$. We know from Theorem 4.1 that $M_1 \times M_2$ admits a generalized complex structure
\begin{align}
\mathcal J_1 (X_1 +\alpha_1,X_2 +\alpha_2 ) = &( \Phi_1(X_1+\alpha_1) - 2\langle E_{+,2},X_2 +\alpha_2 \rangle E_{+,1} - 2\langle E_{-,2}, X_2 +\alpha_2 \rangle E_{-,1}, \cr
 &\Phi_2(X_2 +\alpha_2)  + 2\langle E_{+,1}, X_1 +\alpha_1 \rangle E_{+,2} + 2\langle E_{-,1}, X_1 +\alpha_1 \rangle E_{-,2} ).
\end{align}
Thus it is enough to produce a second generalized complex structure that commutes with $\mathcal J_1$.
Note that $\mathcal J_1$ and $G$ commute by direct computation. So define $\mathcal J_2 = G\mathcal J_1$.  Then $\mathcal J_2^2 = - {\rm Id}$ and $\mathcal J_1 \mathcal J_2 = \mathcal J_2 \mathcal J_1$.

One can now compute an explicit formula for $\mathcal J_2$:
\begin{align}
\mathcal J_2 (X_1 +\alpha_1,X_2 +\alpha_2 ) = &(G \Phi_1(X_1+\alpha_1) - 2\langle E_{-,2},X_2 +\alpha_2 \rangle E_{+,1} - 2\langle E_{+,2}, X_2 +\alpha_2 \rangle E_{-,1},\cr
 & G\Phi_2(X_2 +\alpha_2)  + 2\langle E_{-,1}, X_1 +\alpha_1 \rangle E_{+,2} + 2\langle E_{+,1}, X_1 +\alpha_1 \rangle E_{-,2} ).
\end{align}

A direct calculation shows that $\mathcal J_2^* = -\mathcal J_2$.  All that remains to be shown is that the $\sqrt{-1}$ eigenspaces of $\mathcal J_2$ are closed under the ($\tilde{H}$-twisted) Courant bracket.

From the formula for $\mathcal J_2$ we see that the generators of its $\sqrt{-1}$ eigenspace are given by
\begin{align}
&(E^{(1,0)}_{G_{1}\Phi_1},0) \cr
&(0, E^{(1,0)}_{G_{2}\Phi_2}) \cr
&(E_{+,1}, -\sqrt{-1} E_{+,2}) \cr
&(E_{-,1}, -\sqrt{-1} E_{-,2}).
\end{align}

So it is enough to verify that these generators are closed under the ($\tilde{H}$-twisted) Courant bracket. Since $G_{1}\Phi_{1}$ is strong, we have

$$[[ (E^{(1,0)}_{G_{1}\Phi_1},0),(E^{(1,0)}_{G_{1}\Phi_1},0)]]_{\tilde{H}} = ([[E^{(1,0)}_{G_{1}\Phi_1}, E^{(1,0)}_{G_{1}\Phi_1}]]_{H_1},0) \subset (E^{(1,0)}_{G_{1}\Phi_1},0)$$ by Lemma 4.2 and (4.1).

Similarly,

$$[[(E^{(1,0)}_{G_{1}\Phi_1},0), (E_{\pm,1}, -\sqrt{-1}E_{\pm,2})]]_{\tilde{H}} = ([[E^{(1,0)}_{G_{1}\Phi_1}, E_{\pm,1}]]_{H_1},0) \subset (E^{(1,0)}_{G_{1}\Phi_1},0)$$


and
$$[[(0,E^{(1,0)}_{G_{2}\Phi_2}), (0,E^{(1,0)}_{G_{2}\Phi_2})]]_{\tilde{H}} = (0, [[E^{(1,0)}_{G_{2}\Phi_2}, E^{(1,0)}_{G_{2}\Phi_2}]]_{H_2}) \subset (0,E^{(1,0)}_{G_{2}\Phi_2}).$$

Furthermore,
$$[[(0,E^{(1,0)}_{G_{2}\Phi_2}), (E_{\pm,1}, -\sqrt{-1}E_{\pm,2})]]_{\tilde{H}} = (0, [[E^{(1,0)}_{G_{2}\Phi_2}, -\sqrt{-1}E_{\pm,2}]]_{H_2}) \subset (0,E^{(1,0)}_{G_{2}\Phi_2}).$$

Since $[[E_{\pm, i},E_{\mp ,i}]]_{H_i} = 0$, it is straightforward to compute that $$[[ (E_{+,1} -\sqrt{-1}E_{+,2}) , (E_{-,1}, -\sqrt{-1}E_{-,2})]]_{\tilde{H}} = (0,0)$$ and so the $\sqrt{-1}$ eigenbundle of $\mathcal J_{2}$ is Courant closed and thus $(M_{1}\times M_{2},\mathcal J_{1}, \mathcal J_{2},G)$ is generalized K\"ahler.\\

Conversely, assume $M_1 \times M_2$ is a ($\tilde{H}$-twisted) generalized K\"ahler manifold with ($\tilde{H}$-twisted) generalized complex structures $\mathcal J_1$ and $\mathcal J_2$ as given above. We must show $(\Phi_i, E_{\pm,i}, G_i )$ are $(H_{i}$-twisted) generalized coK\"ahler for $i=1,2.$ By applying Theorem 4.1 to $(M_1 \times M_2, \mathcal J_1)$ we get immediately
that $( \Phi_i, E_{\pm,i}, G_i )$ are normal for $i=1,2.$ Since $\mathcal{J}_2=G\mathcal{J}_1$ is induced from $G_{i}\Phi_i$, we can apply Theorem 4.1 again to $(M_1 \times M_2,  \mathcal J_2) $ and this shows $G_{i}\Phi_i$ are normal. Therefore, $(\Phi_i, E_{\pm,i},G_i)$ is a ($\tilde{H_i}$-twisted) generalized coK\"ahler structure for $i=1,2.$

\end{proof}

Here is another proof of Proposition 3.5 as an application of our main theorem.
\begin{corollary} Any coK\"ahler manifold is generalized coK\"ahler.
\end{corollary}
\begin{proof}
Let $(\phi, \xi, \eta, g)$ be a coK\"ahler structure on $M$ and let $\mathbb{R}$ have its trivial coK\"ahler structure.
Now $M\times \mathbb{R}$ with its product metric is K\"ahler (see Example 2.3). Therefore it is generalized K\"ahler. By applying Theorem 1.3, this gives $(M,\phi, \xi, \eta, g)$ is generalized coK\"ahler.
\end{proof}

\section{Some Examples of Generalized CoK\"ahler Structures}
We have already seen that every classical coK\"ahler structure gives a generalized coK\"ahler structure.
In this section, we provide many more examples of generalized coK\"ahler structures on manifolds. The examples we construct arise from two
general constructions:   $i)$ deformations of generalized K\"ahler structures and $ii)$ products of manifolds. \\
\indent First, we show that the $B$-field transformation of a generalized coK\"ahler structure is again a generalized coK\"aher structure.
\begin{example}

\indent Consider the ($H$-twisted) generalized coK\"ahler structure $(\Phi_{\phi}, E_{\pm}, G)$ and let $B$ be a closed two form, $\Omega$ the fundamental two form as defined in (2.1), and $\pi$ the bivector field from Example 2.14.  Perform
$B$-field transformations obtaining
$$\Phi^{B}_{\phi} = \left ( \begin{array}{cc}  \phi & 0 \\ B\phi + \phi^{*}B & -\phi^{*} \end{array} \right ),  \Phi^{B}_{\Omega} = \left ( \begin{array}{cc}  -\pi^{\sharp}B & \pi^{\sharp} \\ \Omega^{\flat}-B\pi^{\sharp}B & B\pi^{\sharp} \end{array} \right )$$
and the generalized metric given by
$$G^{B} = \left ( \begin{array}{cc}  -g^{-1}B & g^{-1} \\ g-Bg^{-1}B & Bg^{-1} \end{array} \right ). $$
Observe that $[[e^{B}\xi,e^{B}\eta]]_{H}=[[\xi+B\xi,\eta]]_{H}=0$.
Furthermore it can be easily calculated that $G^{B}\Phi^{B}_{\phi}=\Phi^{B}_{\Omega}.$
Since the ($H$-twisted) Courant bracket is invariant under $B$-field transformations, $(\Phi^{B}_{\phi},e^{B}E_{\pm},G^{B})$ is again
($H$-twisted) generalized coK\"ahler.
\end{example}

\indent Recall from Example 2.2 we considered the product of a K\"ahler manifold $(N,J,g)$ and $\mathbb{R}$ (or $S^{1}.$) Using the trivial
coK\"ahler structure on $\mathbb{R}$, one can construct on $N\times \mathbb{R}$ a coK\"ahler structure. This
product construction can be extended to the generalized context, providing a source of many examples.
\begin{proposition}
Let $(M,\mathcal J_{1}, \mathcal J_{2}, G_M)$ be a ($H_{M}$-twisted) generalized K\"ahler manifold and let $(N,\Phi_{1}, E_{N,\pm}, G_N) $
be a ($H_{N}$-twisted) generalized coK\"ahler manifold. Then $M\times N$ admits a ($H_{M\times N}$-twisted) generalized coK\"ahler structure.
\end{proposition}

\begin{proof}
Recall that $T(M\times N) \oplus T^*(M\times N) \approx (TM \oplus T^{*}M )\oplus (TN \oplus T^*N)$
Define the endomorphism $\Phi$ on $(T^*M \oplus TM )\oplus (TN \oplus T^*N )$ by$$\Phi = (\mathcal J, \Phi_1).$$ Define $E_+ = (0, E_{N,+})$, and $ E_- =(0, E_{N,-})$.  Let $G= G_M\times G_N$ be the product metric. It is easy to
verify that $(\Phi, E_{\pm},G)$ is a generalized almost contact metric structure on $M\times N.$ Let $L$ denote the $\sqrt{-1}$ eigenbundle of $\mathcal J$. Then $L^\pm_\Phi  =  (L, E^{(1,0)}_{\Phi_1} )\oplus L_{( 0, E_\pm)}$ is clearly closed under the ($H_{M\times N}$ twisted) Courant bracket which implies that $\Phi $ is strong. Also, observe that $[[E_{+},E_{-}]]_{H_{M\times N}}=0.$  Hence, $(\Phi, E_{\pm}, G)$ is a normal generalized contact structure.
Similarly $L^\pm_{G\Phi}$ is easily seen to be closed under the ($H_{M\times N}$ twisted)Courant bracket so $G\Phi$ is strong.
Therefore, $(M \times N, \Phi ,E_{\pm}, G)$ defines a ($H_{M\times N}$twisted) generalized coK\"ahler structure.
\end{proof}
\indent In \cite{[GOT]}, Goto proves a stability theorem for generalized K\"ahler structures on a manifold $M$ under the hypothesis that there exists an analytic family of generalized complex structures on the manifold. Further, Goto shows that the space of obstructions to deformations of generalized complex structures vanishes in the case of a compact K\"ahler manifold with a holomorphic Poisson structure $\beta$.  We can use these theorems in combination with the above product theorem to construct examples of nontrivial
generalized coK\"ahler manifolds. By \textit{nontrivial generalized coK\"ahler}, we mean that the generalized coK\"ahler structure
does not come from a classical coK\"ahler structure or a B-field transform of a classical coK\"ahler structure.
We first state Goto's stability theorem:
\begin{theorem}\cite{[GOT]}
Let $M$ be a compact K\"ahler manifold of dimension $n.$ If we have an action of an $l$ dimensional complex commutative Lie group $G$ with a
non-trivial $2$-vector $\beta$, then we have a family of deformations of nontrivial generalized K\"ahler structures on $M.$
\end{theorem}

\indent  We combine this theorem with Proposition 5.2 to construct nontrivial generalized coK\"ahler
structures.
\begin{example}
Let $M$ be any compact toric K\"ahler manifold so that the hypothesis of Goto's theorem is satisfied.  Then $M$ then admits a nontrivial
generalized K\"ahler structure. Equip $S^1$ with the trivial generalized coK\"ahler structure given by $$\Phi=0,\hspace{.2cm} E_{+}=\del_t,\hspace{.2cm}E_{-}=dt,\hspace{.2cm}
G = \left ( \begin{array}{cc}  0 & g^{-1}_{S^1} \\ g_{S^1} & 0 \end{array} \right )$$
where $t$ is the coordinate on $S^{1}$ and $g_{S^{1}}=dt\otimes dt.$
Form the product $M\times S^{1}$ which by Proposition 5.2 is generalized coK\"ahler. The nontriviality of the generalized K\"ahler structure on $M$ implies the nontriviality of the generalized coK\"ahler structure on $M\times S^{1}$.
\end{example}

\indent It would be interesting to find examples of strictly almost (co)K\"ahler manifolds that admit a
generalized (co)K\"ahler structure but we were unable to do so except in the twisted case. Gualtieri in \cite{[G1]} showed that the Hopf surface $S^{3}\times S^{1}$, which is nonKahler, does not admit any generalized K\"ahler structure yet it does admit a twisted generalized K\"ahler structure. For additional examples, see for instance (\cite{[ApGu]},\cite{[G3]},\cite{[GOT]}). 

\indent We will construct new examples of almost (co)K\"ahler non-(co)K\"ahler manifolds which admit $H$-twisted generalized (co)K\"ahler structures.
(A remark on terminology:\\ $(M,\omega,J,g)$ is strictly almost K\"ahler if $M$ does not admit any integrable complex structure and $(M,\omega,J,g)$ is almost K\"ahler nonK\"ahler if that particular $J$ is not integrable.) Our examples begin with a construction done by Fino and Tomassini \cite{[FT]} in which they explicitly construct a six-dimensional solvmanifold which admits an $H$-twisted generalized K\"ahler structure. This manifold, denoted by $M^6$,
arises as the total space of a $\mathbb{T}^2$-bundle over the Inoue surface. From the construction,
they are able to compute the first Betti number and show $b_{1}(M^{6})=1.$ Therefore, the manifold $M^{6}$
is a strictly almost K\"ahler manifold which admits an $H$-twisted generalized K\"ahler structure.

\begin{example}(\emph{Twisted Generalized CoK\"ahler of arbitrary odd dimension $ > $ 7})\\
Form the product manifold $M^{7}=M^{6}\times S^{1}$ where $S^{1}$ has the trivial ($H=0$ twisted) generalized coK\"ahler structure and
we get that $M^{7}$ admits an $\tilde{H}$-twisted generalized coK\"ahler structure by Proposition 5.2. Since $b_{1}(M^{6})=1$ \cite{[FT]}, it follows that
$b_{1}(M^{7})=2.$ By Theorem 2.4, this implies $M^{7}$ is strictly almost coK\"ahler. The manifold $M^{6}$ constructed by Fino and Tomassini actually
works in arbitrary even dimension \cite{[FT]} and so denote this manifold by $M^{2n}.$ Now, form the product with $S^{1}$ and define $M^{2n+1}:=M^{2n}\times S^{1}$. By Proposition 5.2, $M^{2n+1}$ admits an $\tilde{H}$-twisted generalized coK\"ahler structure and since $b_{1}(M^{2n+1})=2$, it
is strictly almost coK\"ahler.
\end{example}

The approach we use in the next example follows closely an argument given by Watson \cite{[W]} in
which he constructs higher dimensional almost K\"ahler non-K\"ahler manifolds starting with Thurston's torus
bundle $W^{4}$ over $T^{2}$ which has $b_{1}(W^{4})=3.$

\begin{example}(\emph{Twisted Generalized K\"ahler})\\
Proceeding with the manifold $M^{7}$ in Example 5.5, form the product $M^{14}:=M^{7}\times M^{7}.$ By Theorem 1.3, $M^{14}$ admits an $\widetilde{H_2}$-twisted generalized K\"ahler structure, where $\widetilde{H_{2}}=\widetilde{H_{1}}+\widetilde{H_1}$. Moreover,
$M^{14}$ with this product structure is an almost K\"ahler non-K\"ahler manifold since if it was a K\"ahler manifold then each factor $M^{7}$ would
be coK\"ahler by Theorem 1.1, which is impossible since $b_{1}(M^{7})=2.$ We can continue this process now.
Form the product $M^{15}:=M^{14}\times S^{1}$. This manifold is almost coK\"ahler non-coK\"ahler since if it
was coK\"ahler then $M^{14}$ would be K\"ahler. We can now apply Theorem 1.1
to $M^{22}:= M^{15}\times M^{7}$ concluding that $M^{22}$ is an almost K\"ahler non-K\"ahler manifold that admits an $\widetilde{H}_{3}$-twisted generalized K\"ahler structure. If $M^{22}$ were K\"ahler, then both $M^{15}$ and $M^{7}$ would be
coK\"ahler, which cannot happen. At each iteration, one takes the product with
an $S^{1}$ followed by a product with $M^{7}$. Continuing in this manner, we get $8n+6$-dimensional almost K\"ahler non-K\"ahler manifolds which are $H$-twisted generalized K\"ahler and with $b_{1}=3n+1$, $n=0,1,2,3,...$ Note that
for $n=2k$, the first Betti number is $b_{1}=6k+1$ and so these $16k+6$ dimensional manifolds
are strictly almost K\"ahler manifolds which admit twisted generalized K\"ahler structures. In this procedure we also generate $8n+7$-dimensional almost coK\"ahler non-coK\"ahler manifolds which admit twisted generalized coK\"ahler structures. Moreover, for $n=2m$ the $16m+7$ dimensional manifolds are strictly
almost coK\"ahler manifolds which are twisted generalized coK\"ahler since $b_{1}=6m+2.$ By Theorem 2.4, these manifolds
are strictly almost coK\"ahler.
\end{example}

\begin{remark}
All of these examples are non-diffeomorphic to the examples given by Fino and Tomassini in \cite{[FT]} since our examples
have first Betti number which grows linearly with dimension whereas their examples have $b_1=1$ in arbitrary even dimension.
\end{remark}

An important feature of generalized K\"ahler geometry is its relationship with bi-Hermitian geometry.
It was shown by Gualtieri in \cite{[G1]} that having a generalized K\"ahler structure on a manifold is equivalent to having
a bi-Hermitian structure on the manifold. Therefore, Example 5.6 gives examples of manifolds which admit bi-Hermitian
structures as well.\\

\section*{acknowledgments}
The first author was supported by a James A. Michener Faculty Fellowship at Swarthmore College.
The authors would like to thank Anna Fino for a useful email. The authors would also like to thank the referee for thoughtful suggestions that improved the exposition in the paper.

\end{document}